\documentclass[11pt]{article}
\usepackage{amsfonts,epsf,amsmath,amssymb}
\usepackage{epsfig}
\usepackage{latexsym}
\usepackage{tikz}
\tikzstyle{vertex}=[circle,fill, draw, inner sep=0pt, minimum size=6pt]

\newtheorem{theorem}{\bf Theorem}[section]
\newtheorem{corollary}[theorem]{\bf Corollary}
\newtheorem{lemma}[theorem]{\bf Lemma}

\newtheorem{proposition}[theorem]{\bf Proposition}
\newtheorem{conjecture}[theorem]{\bf Conjecture}

\newcommand{\qed}{\hfill $\square$ \medskip}

\begin{document}

	\title{Finite groups with the same Power graph}
	
	\author{
		M. Mirzargar $\ ^{a}$, R. Scapellato $\ ^{b}$
	}
	
	\date{\today}
	
	\maketitle
	
	\begin{center}
		
		$^a$
		Faculty of Science, Mahallat Institute of Higher Education, Mahallat, I. R. Iran \\
		{\tt m.mirzargar@gmail.com}

		$^b$
		Dipartimento di Matematica, Politecnico di Milano, Milano, Italy\\
		{\tt raffaele.scapellato@polimi.it }

		\medskip
		
		
	\end{center}
	
	\date{}
	
	\maketitle
	
	\begin{abstract}
	The power graph  $P(G)$ of a group $G$ is a graph with vertex set $G$, where two vertices $u$ and $v$ are adjacent if and only if $u \neq v$ and $u^m=v$ or $v^m=u$ for some positive integer $m$. 
 In this paper, we raise and study the following question:
For which natural numbers $n$ every two groups of order $n$ with isomorphic power graphs are isomorphic? In particular, we determine prove that all  such $n$ are cube-free and are not multiples of 16. 
Moreover, we show that if two finite groups have isomorphic power graphs and one of them is nilpotent or has a normal Hall subgroup, the same is true with the other one.
\end{abstract}
	
	
	
	\section{Introduction}

There are many different ways to associate a graph to the given group, including the commuting graphs \cite{ba}, prime graphs \cite{li}, and of course Cayley graphs, which have a long history and applications \cite{k2}. 
Graphs associated with groups and other algebraic structures have been actively investigated since they have valuable applications \cite{k5, x, x2} and specially are related to automata theory \cite{k3, k4}. 
The rigorous development of the mathematical theory of complexity via algebraic automata theory reveals deep and unexpected connections between algebra (semigroups) and areas of science and engineering.  The book  \cite{ito} sets the stage for the application of algebraic automata theory to areas outside mathematics.

Let $G$ be a finite group. The undirected power graph $P(G)$ is the undirected graph with vertex set $G$, where two vertices $a, b \in G$ are adjacent if and only if $a \neq b$ and $a^m = b$ or $b^m = a$ for some positive integer $m$.
Likewise, the directed power graph $\overrightarrow{P}(G)$ is the directed graph with vertex set $G$, where for two vertices $u, v \in G$ there is an arc from $a$ to $b$ if and only if $a\neq b$ and $b = a^m$ for some positive integer $m$.
In \cite{2} you can see a survey of results and open questions on power graphs, also it is explained that the definition given in \cite{k} covers all undirected graphs as
well. This means that the undirected power graphs were also defined in \cite{k} for the first time (see \cite{2, c3, mir2}
for more detailed explanations). These papers used only the brief term ’power graph’, even though they covered both directed and undirected power graphs.
For a group, $G$, the digraph $\overrightarrow{P} (G)$ was considered in \cite{pour} as the main subject of study. 
 In order to measure how close the power graph is to the commuting graph, Aalipour \cite{ali} introduced the enhanced power graph which lies in between. In \cite{feng}, the metric dimension of the power graphs has been studied.
Cameron proved in \cite{c2}, if $G_1$ and $G_2$ are finite groups whose undirected power graphs
are isomorphic, then their directed power graphs are also isomorphic. Clearly, the converse is also true. 
As remarked for instance  in \cite{sen}, $P(G)$ is connected for every $G$ and $P(G)$ is complete if
and only if $G$ is a cyclic group order $1$ or prime-power.
Clearly $G\cong H$ implies $P(G)\cong P(H)$. The converse is false for finite groups in general. 
For example, if $p$ is an odd prime and $m>2$, besides the elementary abelian group $H$ of order $p^m$, there are non-abelian groups $G$ of order $p^m$ and exponent $p$, so $H$ and $G$ are non-isomorphic but have isomorphic power graphs.
On the other hand, it is shown in \cite{c1, mir} that if both $G$ and $H$ are abelian then 
$P(G)\cong P(H)$ implies $G\cong H$. 
Also in \cite{mir}, it is proved that if $G$ is one of the following finite groups: 
\begin{enumerate}
\item  A simple group,
\item  A cyclic group,
\item  A symmetric group,
\item  A dihedral group,
\item A generalized quaternion group,
\end{enumerate}

\noindent and $H$ is a finite group such that $P(G) \cong
P(H)$ then $G \cong H$.

Following \cite{m,s1}, two finite groups $G$ and $H$ are said to be conformal if and only if they have the same number of elements of each order. Such groups need not be isomorphic (see the above example of groups of exponent $p$).
The relevance of this concept to power graphs is due to the fact that, as proved by Cameron \cite{c2}, two
finite groups with isomorphic undirected power graphs are conformal. Note that the converse is not true. For example, two groups of order 16 with the same numbers of elements of each order, e.g. $C_4\times C_4$ and $C_2\times Q_8$ are SmallGroup($16,2$) and SmallGroup($16,4$) in GAP respectively \cite{gap}.  Their power graphs are not isomorphic. In fact, in the group $C_4\times C_4$, each element of order $2$ has four square roots, but in $C_2\times Q_8$, the involution in $Q_8$ has twelve square roots and the other two have none. In \cite{m}, an algorithm is described to find the number of elements of a given order in abelian groups, so if $G$ and $H$ are finite conformal abelian groups, then $G\cong H$.

In \cite{s1}, the following question was investigated:

\noindent{\bf Question:} {For which  natural numbers $n$ every two conformal groups of order $n$ are isomorphic?} 

In \cite{s1}, the set of all such numbers was denoted by $S$ and odd and square-free elements of $S$ were  characterized.

In this paper we raise another question along the same lines: 

\noindent{\bf Question:} {For which natural numbers $n$, every two groups of order $n$ with isomorphic power graphs are isomorphic?}

Let us denote the set of all such numbers by $\bar S$. Since  two
finite groups with isomorphic power graphs are conformal, it is easy to see that $S\subseteq \bar S$.

In this paper we follow the terminology and notation of \cite{harary} for graphs and \cite{rob} for groups. All groups and graphs considered here are finite. In Section $2$, we shall give an answer to the aforementioned question  according to the decomposition of  prime numbers of this number. Moreover, it will be shown that all odd elements of $\bar S$ are cube-free. In Section $3$, we shall prove that if a group is nilpotent, so are all groups having the same power graph. Moreover, we show a similar statement for groups having a normal Hall subgroup.  

\section{Orders of Groups Characterised by Their Power Graphs}

In this section, we study the set $\bar S$, often exploiting methods and results already used for $S$. 
 
In \cite{s1}, Lemma 1, it is proved that if $p$ and $q$ are prime and $q|(p-1)$, then $p^2q\in S$ if and only if $q=2$. Since $S\subseteq \bar S$, the following result is straightforward.

 \begin{proposition}
 If  $p$ is an odd prime number, then $2p^2\in \bar S$.
 \end{proposition} 

Note that $8\in S$, because the two non-abelian groups of order 8 are either the dihedral group $D_8$ or the quaternion group $Q_8$, and the number of elements of order $4$ in these groups is $2$ and $6$, respectively. There are three abelian groups of order $8$, which are pair-wise non-conformal and non-conformal to  $D_8$ or $Q_8$. Therefore $8\in S$ and $8\in \bar S$. 

The following result shows that $\bar S$ contains natural numbers with an arbitrary number of prime factors.

\begin{theorem}\label{prime}
If $n \notin \bar S$ and $(n,k)=1$, then $nk \notin \bar S$.
\end{theorem}
\begin{proof}
Let $G$ and $G'$ be non-isomorphic groups of order $n$ and $P(G)\cong P(G')$. Without loss of generality,
we may assume that $G$ and $G'$ have the same elements and  for each $x\in G$,  $x^r$ is the same in
both $G$ and $G'$, so  their power graphs coincide.  
Let $H$ be a group of order $k$. Note that if $(a,x)$ and $(b,y)$ are elements of $G\times H$, then 
$(b,y)$ is a power of $(a,x)$ if and only if $b$ and $y$ are powers of $a$ and $x$ respectively. Namely, if $b=a^r$ and $y=x^s$, since $(n,k)=1$ we can apply the Chinese Reminder Theorem to the system and 
$ t \equiv r (\text{ mod o(a)})$ and  $t \equiv s (\text{ mod o(x)}) $. We get $(a,x)^t = (a^t,x^t) = (a^r,x^s) = (b,y)$.
The same argument applies to $G'\times H$. Therefore, the power graphs $P(G\times H)$ and $P(G'\times H)$ coincide.
On the other hand, $P(G\times H)$ and $P(G'\times H)$ are non-isomorphic, so  $nk\not\in \bar S$.
\qed
\end{proof}

\begin{lemma}\label{lemma}
 Let $G$ be a $2$-group and  $A$ be an elementary abelian $2$-group. Two vertices $(a,x)$, $(b,y)$ of the graph
$P(G\times A)$ are adjacent if and only if one of the following holds:
\begin{enumerate}
\item $x=y=1$ and $b$ is a power of $a$; 
\item  $x=y\neq 1$ and $b$ is an odd power of $a$;
\item  $x\neq 1$,  $y=1$ and $b$ is an even power of $a$;
\item  $x=1$,  $y\neq 1$ and $a$ is an even power of $b$.
\end{enumerate}
\end{lemma}
\begin{proof} Note first if $x\neq y $ , $x\neq 1$, and $y\neq 1$, then neither $(a,x)$ nor $(b,y)$ can be a power of the other one in the group $G\times  A$. Therefore, we can restrict our attention to ($1$)-($4$) for what $x$ and $y$ are concerned. Besides, ($1$) is obvious. 
If $x=y\neq 1$ and $(a,x)^h=(b,x)$, then $h$ must be odd, hence the vertices are adjacent whenever $b$ is an odd power of $a$; this proves ($2$).
In case ($3$), adjacency is possible only if $(a,x)^h=(b,1)$ for some $h$, and $h$ is necessarily even. Likewise, we get ($4$).
\qed
\end{proof}
\begin{theorem}\label{prime factor}
Let $n=2^{\alpha_0}p_1^{\alpha_1}\cdots p_r^{\alpha_r}$  ($r\geq 0$). If $\alpha_0\geq 4$ or there exists $i\neq 0$ such that $\alpha_i\geq 3$, then $n\not\in \bar S$.  
\end{theorem}
\begin{proof}
We first consider the case $\alpha_0\geq 4$. If $n=2^4$, let $G=K_2\times C_4$ be the direct product of a cyclic group of order $4$ and a Klein four-group and let $H$ be the central product of the  dihedral group of order $8$ and cyclic group of order $4$. 
The groups $G$ and $H$ have IDs $10$ and $13$ respectively among the groups of order $16$ in GAP's SmallGroup library. 
The graphs $P(G)$ and $P(H)$ are isomorphic because both of them, except for the edges incident with $1$, consists of four triangles attached by the only element of order $2$ having square roots. Therefore $n\not\in \bar S$.

Assume that $n=2^{\alpha_0}$ but $\alpha_0 > 4$, and let $A$ be an elementary abelian $2$-group of order $2^{\alpha_0 - 4}$. By Lemma \ref{lemma}, the graphs $P(G\times  A)$ and $P(H\times A)$ are isomorphic, while $G\times  A \ncong  H\times  A$ . Therefore $n\not\in \bar S$.

If $n$ is not a power of $2$, we have $n=2^{\alpha_0}k $, with  $k\neq 1$ odd. Since $2^{\alpha_0}\not\in \bar S$ by the above argument and $(2^{\alpha_0},k)=1$, by Theorem \ref{prime factor} we get $n\not\in \bar S$.  

Let us now consider the case where there exists $i\neq 0$ such that $\alpha_i\geq 3$. Let $G_1$ and $G_2$ be non-isomorphic groups of order $p_i^{\alpha_i}$ and exponent $p_i$. Without loss of generality, we may assume that $G_1$ and $G_2$ have the same elements and that for each $x\in G_1$ the $k$-power of $x$ is the same in both $G_1$ and $G_2$. Let $G'$ be a group of order $\frac{n}{p_i^{\alpha_i}}$. 
For each $(x,a)\in G_1\times G'$, the $k$-power of $(x,a)$ is the same in both $G_1\times G'$ and $G_2\times G'$, therefore $(x,a)$ and $(y,b)$ are adjacent in $P(G_1\times G')$ if and only if they are adjacent in $P(G_2\times G')$. Hence $P(G_1\times G')=P(G_2\times G')$. Since these two groups of order $n$ are not isomorphic, we conclude that 
$n\not\in \bar S$.
  \qed
\end{proof}	
	

\begin{corollary} \label{cor:cube-free}
Every odd element of $\bar S$ is cube-free.
\end{corollary}	

As mentioned above, we have $S\subseteq\bar S$. On the other hand, when we look closely at computer programming, we notice that many small numbers belong to both $S$ and $\bar S$ or to neither. It is then natural to ask whether this
inclusion is indeed strict. 

\begin{theorem}
The set $\bar S\setminus S$ is non-empty. Its smallest element is $72$.
\end{theorem}
\begin{proof}
With the help of GAP's SmallGroup library, we found out that all numbers less than $71$ do not belong to $\bar S\setminus S$.
Applying GAP's SmallGroup library to $72$, we come across with two conformal groups of order $72$ (whose IDs are $35$ and $40$, which we name them $G$ and $G'$ respectively). Therefore, $72 \notin S$.  

Table \ref{a} displays the number of elements of there groups for each possible order. In view of the information provided therein, both groups must have a normal $3$-Sylow subgroup, which is elementary abelian, and nine $2$-Sylow subgroups, which are dihedral. 
The generators and relators provided by GAP lead to the following presentation for $G$:

$<a,b,x,y | a^3=b^3=x^2=y^2=1, (xa)^2=(xb)^2=1, ab=ba, ay=ya, by=yb, (xy)^4=1>.$

The set $M = \{1,(xy)^2,y,xyx\}$ is a subgroup of $G$. Let us prove that each element of $M$ commutes with each element of the $3$-Sylow subgroup $P=<a,b>$. Namely $y$ commutes with $a$ and $b$ because of the relations $ay=ya$ and $by=yb$. Moreover, by using $(xa)^2=1$ and $ay=ay$, we get:
$(xyx)a=xy^{-1}x=x^{-1}yx=a(xyx).$
Likewise, $(xyx)b=b(xyx)$ from $(xb)^2=1$ and $by=yb$. Clearly, $a$ and $b$ must also commute with the remaining elements of $M$. Note that, considering $G$ as a semidirect product of $P$ by $Q$ through a homomorphism 
$f: D\rightarrow \text{Aut}(P)$, the kernel of $f$ is $M$.
The group $C=<a,b,M>$ has order $36$. The products of the $3$ elements of order $2$ in $M$ by the $8$ elements of order $3$ in $<a,b>$ give exactly $24$ elements of order $6$, hence $C$ contains all the elements of order $6$.
In view of the above remarks and of Table \ref{a}, the elements of $G$ are distributed as follows:
in $C$, besides the unity, there are all the elements of order $3$ and of order $6$, plus three elements of order $2$ (in $M$);
in $G \setminus C$ there are all the elements of order $4$ plus the remaining eighteen elements of order $2$.
The only elements whose roots are of interest for the power graph $P(G)$, are those of order $2$ and of order $3$.
Since $M$ has index $2$ in every $2$-Sylow subgroups, it is normal and the intersection of two of them contains it but cannot be larger, hence the nine $2$-Sylow subgroups pair-wise intersect in $M$. Therefore $(xy)^2$ has as roots all the elements of order $4$, while none of the remaining elements of order $2$ has roots. Each element of order $3$ has as roots the three elements of order $6$ obtained by multiplying it by the nontrivial elements of $M$.

\begin{table}[]\label{a}
\begin{center}
\begin{tabular}{|l|l|l|l|l|l|l|l|l|l|l|l|l|}
\hline
 order  & 1  & 2 & 3 & 4 & 6 & 8 & 9 & 12 & 18 & 24 & 36 & 72 \\
 \hline
 number & 1 & 21 & 8 & 18 & 24 & 0 & 0 & 0 & 0 & 0 & 0 & 0  \\ 
\hline
\end{tabular}
\caption{The number of elements of each order in SmallGroup($72,35$) and SmallGroup($72,40$)  }
\end{center}
\end{table}
 
Now, we come to the group $G'$ whose ID in SmallGroup library is $40$.
The generators and relators provided by GAP lead to the following presentation for $G'$:

$G' = <a,b,x,y,z |  a^3=b^3=x^2=y^2=z^2=1, (xy)^2=z, ab=ba, xax=a^{-1}, xbx=b, yay=b, yby=a, zaz=a^{-1}, zbz=b^{-1}>.$

About the subgroups of $G'$ we have the group $P=<a,b>$ which is the unique $3$-Sylow subgroup of  $G'$,  $D=<x,y>$ is one of the $2$-Sylow subgroups of $G'$ and dihedral group. As in the case of $G$, this group is a semidirect product of $P$ by $D$ through a homomorphism $f': D \rightarrow \text{Aut}(P)$. This time, $f'$ is injective. Because the three automorphisms induced by $x, y$ and $z$ generate a group of order $8$. 
There are many ways to see that the power graph of $G'$ cannot be isomorphic with that of $G$. For instance, unlike what happens in $G$, there are no elements of order $2$ of $G'$ whose roots are all the elements of order $4$.
There are no elements of order $2$ of $G'$ with $8$ roots of order $3$. 
Each element of order $3$ has three roots of order $6$, but unlike what happens with $G$, the cubes of such root depend on the element.
For example, for $a$ we get $a^{-1}yxy$, $ab^{-1}yxy$, $a^{-1}b^{-1}yxy$, whose cubes are equal to $yxy$; for $b$  we get $b^{-1}x$, $ab^{-1}x$, $a^{-1}b^{-1}x$, whose cubes are equal to $x$.
Therefore, $P(G')$ cannot be isomorphic with $P(G)$, so $72\in\bar S $. \qed
\end{proof}

\begin{conjecture}
The set $\bar S \setminus S$ is infinite.
\end{conjecture}

\section{Power Graphs of Nilpotent Groups and Groups Having a Normal Hall Subgroup}

Again exploiting the necessary condition of conformality, we are going to show here some situations where a property of a group $G$ is inherited by all groups with the same power graph.

\begin{theorem} If $G$ and $H$ are conformal and $H$ is nilpotent, then also $G$ is nilpotent. 
\end{theorem}
\begin{proof}Since $G$ and $H$ are conformal, for each prime $p$ dividing the common group order, 
$G$ has the same number of
elements of $p$-power order as $H$. Since $H$ is nilpotent, the number of elements
of $p$-power order is equal to the order of the $p$-Sylow subgroups of $G$ and $H$.
Thus $G$ contains only one $p$-Sylow subgroup for each prime dividing the
group order, hence is nilpotent. \qed
\end{proof}

\begin{corollary}
If $P(G)\cong P(H)$ and $H$ is nilpotent, then also $G$ is nilpotent. 
\end{corollary}

A subgroup of a finite group is said to be a Hall subgroup if its order and index are relatively prime.

\begin{theorem}Let $G$ and $H$ be conformal groups. If $H$ has a normal Hall subgroup of order $m$ and $G$ is solvable, then also $G$ has a normal Hall subgroup of order $m$. 
\end{theorem}
\begin{proof}Since $G$ and $H$ are conformal, the elements of $H$ whose order divides
$m$ are exactly $m$, hence the same happens for $G$. Since $G$ is solvable and $m$ is prime with $|G|/m$, by Hall's Theorem it must have a Hall subgroup of order $m$, which is normal because it contains all the elements whose order divides $m$. \qed
\end{proof}
\begin{corollary}
If $P(G) \cong P(H)$, $H$ has a normal Hall subgroup of order $m$, and $G$ is solvable, then also $G$ has a normal Hall subgroup of order $m$. 
\end{corollary}

\section*{Conclusion}
There is not a one to one function between groups and power graphs.  Therefore, the power graphs do not always determine the groups. 
 An interesting study would be to find out 
for which groups $G$ and $H$,  $P(G)\cong P(H)$ implies $G\cong H$.
The  present paper aims  to  classify power graphs based on group orders, which can be a new look at the power graphs classification. 
Moreover, the concept of conformal groups and the order of the elements of a group play an important role in the results of this paper and guide us to classify power graphs of nilpotent groups and groups which have a normal Hall subgroup.
 The authors believe that it is possible to classify power graphs based on the order of their groups. This topic can continue and leads many  open questions motivated by classification problems for future work.

\end{document}